\documentclass[11pt, twoside, letterpaper]{amsart}

\usepackage{t1enc}
\usepackage{latexsym}
\usepackage{amssymb}
\usepackage{graphicx}
\usepackage{amsmath}
\usepackage{amsthm}
\usepackage{amsfonts}

\usepackage{mathrsfs}

\usepackage[all]{xy}
\usepackage[british]{babel}

\usepackage{color}
\usepackage[hypertexnames=false,
    pdftex,
	pdfpagemode=UseNone,
	breaklinks=true,
	extension=pdf,
	colorlinks=true,
	linkcolor=blue,
	citecolor=blue,
	urlcolor=blue,
]{hyperref}

\newcommand{\cyclic}{\mathop{\kern0.9ex{{+}\kern-2.10ex\raise-0.20
      ex\hbox{\Large\hbox{$\circlearrowright$}}}}\limits}

\newtheoremstyle{daniel}{3.0mm}{2mm}{\itshape}{}{\bfseries}{.}{1.5mm}{}
\theoremstyle{daniel}
\newtheorem{thm}{Theorem}[section]
\newtheorem{prop}[thm]{Proposition}
\newtheorem{Defi}[thm]{Definition}
\newtheorem{lemma}[thm]{Lemma}

\newtheorem{Exs}[thm]{Examples}
\newtheorem{Rems}[thm]{Remarks}

\newtheorem*{thm*}{Theorem}
\newtheorem*{cor*}{Corollary}
\newtheorem*{thm3.6}{Theorem 3.6}
\newtheorem*{thm4.3}{Theorem 4.3}

\newtheorem*{prop*}{Proposition}
\newtheorem*{Notation}{Notation}

\newtheorem{claim}[thm]{Claim}
\newtheorem{Def}[thm]{Definition}

\newtheorem{Rem}[thm]{Remark}

\newtheorem{Ex}[thm]{Example}

\newtheorem*{Setup}{Setup}

\def\cC{\mathcal C }

\def\cF{\mathcal F}
\def\cH{\mathcal H}

\def\cP{\mathcal P}

\def\cO{\mathcal O}

\def\cT{\mathcal T}

\def\gm{{\mathfrak m}}

\newcommand{\C}{\mathbb{C}}
\newcommand{\D}{\mathbb{D}}
\newcommand{\N}{\mathbb{N}}

\newcommand{\R}{\mathbb{R}}
\newcommand{\Z}{\mathbb{Z}}

\DeclareMathOperator{\Coh}{Coh}
\def\ch{\mbox{ch}}
\def\Coker{\mbox{Coker}}
\DeclareMathOperator{\cycle}{cycle}
\def\d{\mbox{d}}
\def\gr{\mbox{gr}_{JH}}
\def\Hom{\mbox{Hom}}
\def\im{\mbox{Im}}
\def\id{\mbox{id}}
\def\Ker{\mbox{Ker}}
\def\NS{\mbox{NS}}

\def\Spec{\mbox{Spec}}

\def\Tors{\mbox{Tors}}

\DeclareMathOperator{\rk}{rk}

\numberwithin{equation}{section}

\begin{document}
\title[families of coherent analytic sheaves]{Properness criteria for families of coherent analytic sheaves}

\author{Matei Toma}
\address{Matei Toma, Universit\'e de Lorraine, CNRS, IECL, F-54000 Nancy, France
}
\email{Matei.Toma@univ-lorraine.fr}

\date{\today}
\keywords{semistable coherent sheaves, moduli of sheaves, properness criteria}
\subjclass[2010]{32G13, 14D20}

\begin{abstract}
We extend Langton's valuative criterion for families of coherent algebraic sheaves to a complex analytic set-up. As a consequence we derive a set of sufficient conditions for the compactness of moduli spaces of semistable sheaves over  compact complex manifolds. This  applies also to some cases appearing in complex projective geometry not  covered by previous results.  
\end{abstract}
\maketitle
\setcounter{tocdepth}{1}
\tableofcontents
\noindent


\section{Introduction}

There is a variety of situations when moduli spaces of semistable coherent sheaves over projective schemes are known to exist, see \cite{HL} for an extensive treatment of this topic. 
Such moduli spaces typically turn out to be projective. This is often proved by using a result of Langton who checked that Grothendieck's valuative criterion for properness applies to families of semistable sheaves   \cite{Langton}, \cite[Theorem 2.B.1]{HL}. 
Sometimes even if the base variety $X$ is projective, in order to define stability one may be led to consider arbitrary real ample classes as polarizations, cf.  \cite{GrebToma}, \cite{GRT1}, \cite{GRTsurvey},  \cite{CampanaPaun}. However when the polarization is irrational Langton's result doesn't directly apply and its proof cannot be adapted in a straightforward way to such a case.

When the base variety is compact complex analytic, one may still speak of semistability with respect to  Gauduchon metrics. But even over compact K\"ahler manifolds an analogous result to Langton's is not available and complex analytic moduli spaces for semistable sheaves have only been constructed in special cases. 

It is the purpose of this paper to provide replacements of properness valuative criteria as in Langton's result in a complex analytic set-up and to show how they may be applied to prove compactness of moduli spaces of semistable sheaves.  
For coherent sheaves over a compact complex manifold a fairly general semistability notion will be described in Section \ref{section:degrees}. (We define semistability only for pure sheaves but we later give analogues of our statements in the non-pure case.) 
Then in Section \ref{section:one-dimensional-case} we deal with families of semistable sheaves over one dimensional bases. In this context we prove Theorem \ref{thm:dimension one} an analogue of which over spectra of discrete valuation rings already provides a solution to the case of irrational polarizations mentioned above. This theorem basically says that for a flat family of coherent sheaves with general semistable members over a curve one may replace the special members in such a way that the resulting family has only semistable fibres. 
When the base manifold is not projective however, such a result on one dimensional families may be no longer sufficient. 
This is related to the fact that the two current definitions of meromorphic mapppings do not coincide. Stoll's definition which uses extensions along curves is weaker than Remmert's which requires a factorization through a proper modification, \cite[pp. 424-428]{StollMeromorpheAbbildungenII}, see also   \cite{Hirschowitz}.  
In Section \ref{section:arbitrary-dimension} we prove Theorem \ref{thm:dimension>1} which provides a replacement for Langton's valuative criterion also for families over higher dimensional bases, albeit under a more restrictive semistabilty condition. Using this result we give in Section \ref{section:application} sufficient conditions for a moduli space of semistable sheaves over a compact complex manifold to be compact. The paper ends with a criterion for separatedness (Theorem \ref{thm:separation}).

For Theorem \ref{thm:dimension one} we follow Langton's original line of proof with two new inputs, one of combinatorial nature which compensates the lack of rationality of the polarizations involved, and one application of Artin approximation which avoids non-properness issues of the relative Douady space. Theorems   \ref{thm:dimension>1} and \ref{thm:application} 
are essentially new and they are especially pertinent to the complex analytic context. Their proofs depend on a notion of boundedness for sets of isomorphy classes of coherent analytic sheaves, which was introduced in \cite{TomaLimitareaI}.  

{\em Acknowledgements:} I wish to thank Sergei Ivashkovich for pointing out to me Hirschowitz' paper \cite{Hirschowitz} which discusses the two definitions of meromorphic mappings.  


\section{Degree functions and stability}\label{section:degrees}
Various definitions of semistability for coherent sheaves on projective manifolds are in use and many recent papers aim at a formalization of their properties, see e.g. \cite{JoyceIII}, \cite{Andre}.

Here we content ourselves with the presentation of the stability notion which will appear in the results of this paper. 
It will use a generalization of the Hilbert polynomial of a coherent sheaf which will make sense in a complex geometric (not necessarily projective) framework. This notion finds applications even when the base space $X$ is a smooth projective variety; cf. \cite{GrebToma}, \cite{GRT1}, \cite{GTgood}. In order to introduce it one may choose to work either on a category $\Coh_d(X)$ of coherent sheaves of dimension at most $d$ on an analytic space $X$ or on a quotient category $\Coh_{d,d'}(X):=\Coh_d(X)/\Coh_{d'-1}(X)$ as in \cite[Section 1.6]{HL}. In this paper we chose the first but the statements can be easily rephrased to stay valid for the second approach.  

In the sequel $X$ will be a compact analytic space of dimension $n$ and $d$, $d'$ will denote integers satisfying 
$n\ge d\ge d'\ge 0$. In particular $X$ may be the associated  analytic space of a proper algebraic space over $\C$.

We denote by $K_0(X)=K_0(\Coh(X))$ the Grothendieck group of coherent sheaves on $X$ and by $[F]$ the class in $K_0(X)$ of a coherent sheaf $F$.
If $F$ has dimension at most $p$, we write $\cycle_p(F)$ for the $p$-cycle associated to $F$.  

\begin{Def} Degrees.
\label{def:degrees}
Consider a group morphism 
$\deg_p:K_0(X)\to\R$ inducing a positive map on non-zero $p$-cycles when putting $\deg_p(Z):=\deg_p([\cO_Z])$ for irreducible $p$-cycles $Z$, and the following properties:
\begin{enumerate}
\item $\deg_p([F])=\deg_p(\cycle_p(F))$ for any $F\in\Coh_p(X)$,
\item if a set of positive $p$-cycles is such that $\deg_p$ is bounded on it, then $\deg_p$ takes only finitely many values on this set, 
\item $\deg_p$ is 
continuous
on flat families of sheaves.
\item $\deg_p$ is locally constant 
on flat families of sheaves.
\end{enumerate}
We will call such a function a {\em degree function in dimension $p$} if it has properties (1), (2) and (3).
If only the first property is satisfied, we will call it a {\em weak degree function in dimension $p$} and if $\deg_p$ has all four properties, we will call it a {\em strong degree function in dimension $p$}.
We will write for simplicity $\deg_p(F)=\deg_p([F])$ for any $F\in \Coh(X)$.

A collection of  degree functions $(\deg_d,...,\deg_{d'})$ in dimensions $d$ to $d'$ on $X$ will be called a {\em  $(d,d')$-degree system} and similarly for weak or strong degree functions.
\end{Def}

Note that for any weak degree function $\deg_p$ in dimension $p$ on $X$ one has $\deg_p(F)>0$  if $F$ is $p$-dimensional, and $\deg_p(F)=0$ if $F\in\Coh_{p-1}(X)$.

Strong degree functions appear naturally if $X$ is endowed with  differential forms $\omega_{p}$ of degree $2p$ which are $\d$-closed and such that their   $(p,p)$-components $\omega_p^{p,p}$ are strictly positive in Lelong's sense; cf. \cite[III.2.4, IV.10.6]{BarletMagnusson}. In this case for $F\in \Coh(X)$ one defines   
$\deg_p(F):=\int_{\tau_p(F)}\omega_p$, where $\tau(F)$ is the homological Todd class of $\cF$ and the integral is computed on a semianalytic representative of $\tau_p(F)$, cf. \cite[Section 2]{TomaLimitareaI},  \cite{BlHe69}, \cite{DoPo77}, \cite[Thm. 7.22]{AG}, \cite[8.4]{Gor81}, \cite{Her66}. 
These functions satisfy condition (2) of Definition \ref{def:degrees} since they are constant on connected components of the corresponding cycle spaces and since no analytic space can have an infinite number of irreducible components accumulating at a point.
When $(X,\omega)$ is a K\"ahler space, cf. \cite[II.1.2]{Var89}, we obtain such $2p$-forms as $p$-powers of the K\"ahler form $\omega$. When
 $(X, \cO_X(1))$ is a projective variety endowed with an ample line bundle, by taking $\omega$ a strictly positive curvature form of $\cO_X(1)$, we recover the coefficient of the Hilbert polynomial of $F$ in degree $p$ as 
$$
\frac{\deg_p(F)}{p!}.
$$  
However even  on projective manifolds one is naturally led to consider degree functions which are not associated to an ample polarization, cf. \cite{GrebToma}, \cite{GRT1}, \cite{GRTsurvey},  \cite{CampanaPaun}, \cite{GKP-Movable}. 

In the above situation the condition $\d\omega^{p,p}_p=0$ implies that the corresponding degree function is locally constant on flat families of sheaves. An example having found applications in the literature, where only continuity holds is that of degree functions on non-K\"ahlerian  compact manifolds. Such a manifold always carries a Gauduchon form $\omega_{n-1}$, i.e. such that $\omega_{n-1}$ is positive of type $(n-1,n-1)$ and $\partial\bar\partial\omega_{n-1}=0$. One then defines degree functions in dimensions $n$ and $n-1$ by setting $\deg_n(F):=\rk(F)$ and $\deg_{n-1}(F):=\int_X [\omega_{n-1}]_A\cdot c_{1}(F)_{BC}$, where the classes are computed in Aeppli cohomology and in Bott-Chern cohomology respectively, cf. \cite{LuTe}, \cite{Tel10}.
More generally, any strictly positive $\partial\bar\partial$-closed $(p,p)$-form on $X$ gives rise to a degree function in dimension $p$. One shows that 
 Condition (2) of Definition \ref{def:degrees} is  satisfied by the same argument as in the K\"ahler case. Indeed such a  function is pluriharmonic on the cycle space by \cite[Proposition 1]{BarletConvexitate} and attains its minimum on any closed subset of the cycle space by Bishop's theorem. It is therefore constant on any irreducible component of this space. Note that any compact complex manifold $X$ admits a degree function in dimension zero defined by $\deg_0(F):=\int_X\ch_n(F)$. In this way Gauduchon surfaces $(X,\omega_1)$ get a $(2,0)$-degree system.

 For the following definition the order relation we will consider on $\R^{d-d'+1}$ will be   the lexicographic order.

\begin{Def}Semistability.

Suppose that $X$ is equipped with a weak $(d,d')$-degree system $(\deg_{d},...,\deg_{d'})$. For any $d$-dimensional coherent sheaf $F$ we define its slope vector with respect to this system as
$$\mu(F):=(\frac{\deg_{d}(F)}{\deg_d(F)},\frac{\deg_{d-1}(F)}{\deg_d(F)},..., \frac{\deg_{d'}(F)}{\deg_d(F)})\in \R^{d-d'+1}.$$
A $d$-dimensional sheaf $F$ will be called {\em slope-semistable} or just {\em semistable} if  it is pure and if for any non-trivial subsheaf $E\subset F$ we have $\mu(E)\le\mu(F)$.
\end{Def}

Note that in the case $d'=d$,  semistability on $\Coh_d(X)$ just means $d$-dimensional purity.

With literally the same proof as in \cite[Section 1.3]{HL} one checks the existence of a Harder-Narasimhan filtration for the above semistability notion:

\begin{thm} With respect to a weak $(d,d')$-degree system on $X$ any pure $d$-dimensional sheaf $F$ admits a unique increasing filtration
$$0=HN_0(F)\subset  HN_1(F)\subset...\subset HN_l(F)=F,$$
with semistable factors $HN_i(F)/HN_{i-1}(F)$ for $1\le i\le l$ and such that
$$\mu( HN_1(F)/HN_{0}(F))>...>\mu( HN_l(F)/HN_{l-1}(F)).$$
\end{thm}

In particular under the above hypotheses $HN_1(F)$ has the properties of a {\em maximal destabilizing subsheaf} of $F$, i.e. for all subsheaves $E\subset F$ one has $\mu(E)\le \mu(HN_1(F))$, and in case of equality $E\subset HN_1(F)$. Moreover, $HN_1(F)$ is a {\em saturated subsheaf of $F$}, i.e the quotient $F/HN_1(F)$ is either zero or pure $d$-dimensional.

Before we go on to the relative case let us remark that purity is a Zariski open property in flat families of coherent sheaves. Indeed 
if  $S$ is any analytic space and if $E$ is a flat family of $d$-dimensional coherent sheaves on the fibres of $X$ parametrized by $S$, then one can adapt Maruyama's approach in  \cite[Proposition 1.13]{Maruyama-Construction} to prove that the set of points $s\in S$ such that $E_s$ is not pure is a closed analytic subset of $S$.  It suffices to work in loc. cit. with local resolutions and apply the purity criterion from \cite[Lemma 1.12]{Maruyama-Construction}.

By a
{\em relative Harder-Narasimhan filtration} for a flat family $E$ of $d$-dimensional coherent sheaves on $X$ parametrized by an irreducible analytic space $S$ we mean a proper bimeromorphic morphism of irreducible analytic spaces $T\to S$ together with a filtration 
$$0=HN_0(E)\subset  HN_1(E)\subset...\subset HN_l(E)=E_T$$
such that the factors $HN_i(E)/HN_{i-1}(E)$ are flat over $T$ for $1\le i\le l$ and which induces the absolute Harder-Narasimhan filtrations fibrewise over some dense Zariski open subset of $S$, cf. \cite{TomaLimitareaII} for a more general situation and \cite[Section 2.3]{HL} for the projective algebraic case. 

In order to obtain a relative Harder-Narasimhan filtration for a family of sheaves we need stronger assumptions on the degree functions. Using the techniques of \cite{TomaLimitareaI}  the following result is obtained in \cite{TomaLimitareaII}.

\begin{thm}\label{thm:relativeHNfiltration} Suppose that $X$ is endowed with a strong $(d,d')$-degree system induced by a system of strictly positive $\partial\bar\partial$-closed differential forms.  Then with respect to the corresponding semistability notion every  flat family $E$ of $d$-dimensional coherent sheaves on $X$ with pure general members parametrized by an irreducible analytic space $S$ has a relative Harder-Narasimhan filtration $(T\to S, HN_\bullet(E))$. Moreover this filtration has the following universal property:  if $f:T'\to S$ is a bimeromorphic morphism of irreducible analytic spaces and if $F_\bullet$ is a filtration of $E_{T'}$ with flat factors, which coincides  fibrewise with the absolute Harder-Narasimhan filtration over general points $s\in S$, then $f$ factorizes over $T$ and $F_\bullet=HN_\bullet(E)_{T'}$.
\end{thm}

A consequence of this theorem is the fact that semistability is a Zariski open property in flat families of sheaves.  

Another way to look at the relative Harder-Narasimhan filtration is to consider its direct image over $X\times S$ and the filtration which this induces on $E$.  In particular, if $E$ is a family as in the theorem's statement and whose general fibres are not semistable and  if $(f:T\to S, HN_\bullet(E))$ is the relative Harder-Narasimhan filtration of $E$, then the fibres over general points $s\in S$ of the image $F_1$ of the composition of sheaf homomorphisms 
$$(\id_X\times f)_*(HN_1(E)) \to (\id_X\times f)_*(E_T)\to E$$
coincide with $HN_1(E_s)$. We shall call the sheaf $F_1$ the {\em  relative maximal destabilizing sheaf} of $E$.


\section{One dimensional families}\label{section:one-dimensional-case}
In this section we deal with the case of one-dimensional families in its analytic formulation. The attentive reader will be able to translate the argument in terms of families over the spectrum of a discrete valuation ring, when the base space is algebraic, with some care however when applying Artin approximation. 
We will denote by $\D$ the open unit disc in $\C$ and write $\D^*:=\D\setminus\{0\}$. 

\begin{thm}\label{thm:dimension one} 
Let  $X$ be a compact complex manifold endowed with a $(d,d')$-degree system and let $F$ be a $\D$-flat family of $d$-dimensional sheaves on $X$.  Suppose that for $s\in\D\setminus\{0\}$ the fibres $F_s$ are semistable. Then there exists a coherent subsheaf $F'\subset F$ coinciding with $F$ over $\D\setminus\{0\}$ and such that the fibre $F'_0$ over zero is also semistable.
\end{thm}

Before starting the proof, note that $\cO_{\D,s}$ is a principal ideal domain for any $s\in\D$, so for an $\cO_{\D,s}$-module being flat boils down to being torsion-free. Thus, since $F$ has no $\D$-torsion, any coherent subsheaf of $F$ continues to be flat over $\D$.

\begin{proof}
The proof follows the line of \cite[Section 2B]{HL} with an essential expansion due to the lack of discreteness of the degree functions in our set-up. A smaller change appears at the end where we avoid the issue on the use of the properness of the relative Douady space and replace it by a short argument based on Artin approximation.

For any integer $\delta\in[d',d]$ we will consider semistability with respect to the $(d,\delta)$-degree system $(\deg_d,...,\deg_\delta)$ obtained by restricting the given $(d,d')$-degree system on $X$. We will call a semistable sheaf with respect to such a restricted system shortly {\em $(d,\delta)$-semistable}.
Note that a sheaf is $(d,d)$-semistable if and only if it is pure of dimension $d$. Thus for $d'=d$  Theorem \ref{thm:dimension one} just says that if $F$ is flat and with pure fibres over $\D^*$, then a subsheaf $F'\subset F$ exists with $F'_{\D^*}=F_{\D^*}$ and $F'_0$ pure as well. This is the case $\delta=d$ of the following Claim and a proof is suggested in \cite[Exercise 2.B.2]{HL}. Since this special case works under weaker hypotheses and has an easier proof we provide a separate statement for it as Proposition \ref{prop:dimension one}.

\begin{claim}\label{Claim}
Let $X$, $F$ be as in the theorem's statement and $d\ge\delta \ge d'$. Then, if $F_0$ is moreover $(d,\delta+1)$-semistable, there exists  a subsheaf $F'\subset F$  with $F'_{\D^*}=F_{\D^*}$ and  $(d,\delta)$-semistable fibre $F'_0$ over zero. (Here when $\delta=d$ we think of the $(d,d+1)$-semistability condition as being void, that is all sheaves $F_0$ are automatically $(d,d+1)$-semistable.)
\end{claim}

The theorem will follow from this Claim by descending induction on $\delta$. 
 
We now proceed to the proof of the Claim. The case $\delta=d$ has already been discussed, so we will work under the hypothesis $d>\delta\ge d'$. Assuming by contradiction that the Claim is false we first construct an infinite descending filtration $F=F^0\supset F^1 \supset F^2 \supset F^3 \supset...$ such that $F^n_{\D^*}= F_{\D^*}$ and $F^n_0$ not $(d,\delta)$-semistable for every $n\in\N$, as follows: supposing that $F^n$ has already been constructed, let $B^n\subset F_0^n$ be the maximal $(d,\delta)$-destabilizing subsheaf of $F_0^n$, let $G^n:=F_0^n/B^n$ and let $F^{n+1}$ be the kernel of the composition $F^n\to F_0^n\to G^n$. 
These sheaves are related through two exact sequences
\begin{equation}\label{seq:1}
0\to B^n\to F_0^n\to G^n\to 0
\end{equation}
\begin{equation}\label{seq:2}
0\to G^n\to F_0^{n+1}\to B^n\to 0.
\end{equation}
To explain the second we will tensorize over $\cO_\D$ the exact sequence 
$0\to F^{n+1}\to F^{n}\to G^n\to 0$ by 
$0\to \gm\to \cO_\D\to\cO_\D/\gm\to0$, where $\gm$ is the ideal sheaf of the origin in $\D$, to get the following  commutative diagram with exact rows and columns. 
\begin{equation}
\label{diagr:snake}
\begin{gathered}
\xymatrix{ 
 & & &0\ar[d] & \\
 & 0 \ar[d] & 0 \ar[d]& Tor_1^{\cO_\D}(G^n,\cO_\D/\gm) \ar[d]& \\ 
0 \ar[r]& F^{n+1}  \otimes_\D\gm \ar[r]\ar[d] & F^{n+1} \ar[d]\ar[r] & F_0^{n+1}\ar[r] \ar[d] & 0 \\
0 \ar[r]&  F^{n}\otimes_\D\gm \ar[r]\ar[d] & F^{n} \ar[d]\ar[r] & F_0^{n}\ar[r] \ar[d] & 0 \\
 &  G^n\otimes_\D\gm \ar[r]^0\ar[d] & G^n \ar[d]\ar[r]^{\cong} & G^n\otimes_\D\cO_\D/\gm\ar[r] \ar[d] & 0 \\
 & 0  & 0  & 0 & 
}
\end{gathered}
\end{equation}
Then using the Snake Lemma and the fact that $\gm\cong\cO_\D$, we see that the right vertical exact sequence decomposes into two short exact sequence which are precisely the sequences \eqref{seq:2} and \eqref{seq:1}. These exact sequences appear also in Langton's paper; cf.  \cite[Appendix C]{MasloSepp} for a detailed account of Langton's proof. 

Combining the exact sequences \eqref{seq:1}, \eqref{seq:2} we get a self explanatory commutative diagram with exact rows and columns:
\begin{equation}
\label{diagr:3X3}
\begin{gathered}
\xymatrix{ 
 & 0 \ar[d] & 0 \ar[d] & 0 \ar[d]& \\ 
0 \ar[r]& A^{n+1} \ar[r]\ar[d] & B^{n+1} \ar[d]\ar[r] & C^{n+1}\ar[r] \ar[d] & 0 \\
0 \ar[r]& G^{n} \ar[r]\ar[d] & F^{n+1}_0 \ar[d]\ar[r] & B^{n}\ar[r] \ar[d] & 0 \\
0 \ar[r] & K^{n+1} \ar[r]\ar[d] & G^{n+1} \ar[d]\ar[r] & L^{n+1}\ar[r] \ar[d] & 0 \\
 & 0  & 0  & 0 & 
}
\end{gathered}
\end{equation}
By continuity $\deg_p(F_0^n)=\deg_p(F_0)$ for all $n\in\N$ and $p$ with $d\ge p\ge d'$. 
If $C^{n+1}\neq 0$, then $C^{n+1}$ is pure $d$-dimensional as a  submodule of the $(d,\delta)$-semistable module  $B^{n}$ and 
\begin{equation}\label{ineq:1}
\mu_{d,\delta}(B^{n+1})\le \mu_{d,\delta}(C^{n+1})\le \mu_{d,\delta}(B^{n}).
\end{equation}
If 
$C^{n+1}=0$, we have 
\begin{equation}\label{ineq:2}
\mu_{d,\delta}(B^{n+1})=\mu_{d,\delta}(A^{n+1})\le \mu_{d,\delta}(G^{n})<\mu_{d,\delta}(B^{n})
\end{equation}
by the choice of $B^n\subset F^n_0$. We also have $\mu_{d,\delta}(B^{n})>\mu_{d,\delta}(F_0)$, hence the sequence $(\mu_{d,\delta}(B^{n}))_n$ is descending and bounded from below. In fact only the $\frac{\deg_\delta}{\deg_d}$ may vary on this sequence. At this point if we knew the degree functions to be discrete, we would conclude that the sequence $(\mu_{d,\delta}(B^{n}))_n$ is eventually stationary. In our situation we need to construct further  objects in order to get the eventually stationary behaviour of this sequence. 

The idea is to introduce at each formation step of the subsheaves $F^n$ a decomposition of $F_0^n$ into smaller and smaller building blocks as $n$ increases. We will get at each step a collection $\cC^n$ of $2^{n+1}$ building blocks. We show that this crumbling process must eventually stop, leading to the desired eventually stationary behaviour. We indicate below the first three steps of this process.

Step 0: The decomposition is given by the exact sequence \ref{seq:1} for $n=0$. We set $\cC^0:=(B^0,G^0)$.

Step 1: The case $n=1$ for the sequences  \eqref{seq:1} and \eqref{seq:2} lead to a diagram of type \eqref{diagr:3X3}. We set 
$\cC^1:=(A^1, C^1, K^1, L^1 )$. We may view $\cC^1$ as the result of cutting $\cC^0$ into pieces by using \eqref{seq:2}.
The reconstruction of $B^0,B^1,G^0,G^1,F_0^0,F_0^1$ is possible starting from $\cC^1$.

Step 2: We use again \eqref{seq:2} this time for $n=2$ to cut each component of $\cC^1$ into two further pieces. These give the eight vertices denoted by $\ast$ in the following diagram. We allow to count isomorphic components of $\cC^n$ several times if they appear at different places in the decomposition process. Note that $B^2$ cuts off sub-objects of components of $\cC^1$; these appear represented in the vertical plane containing $B^2$ (the back plane).
$$
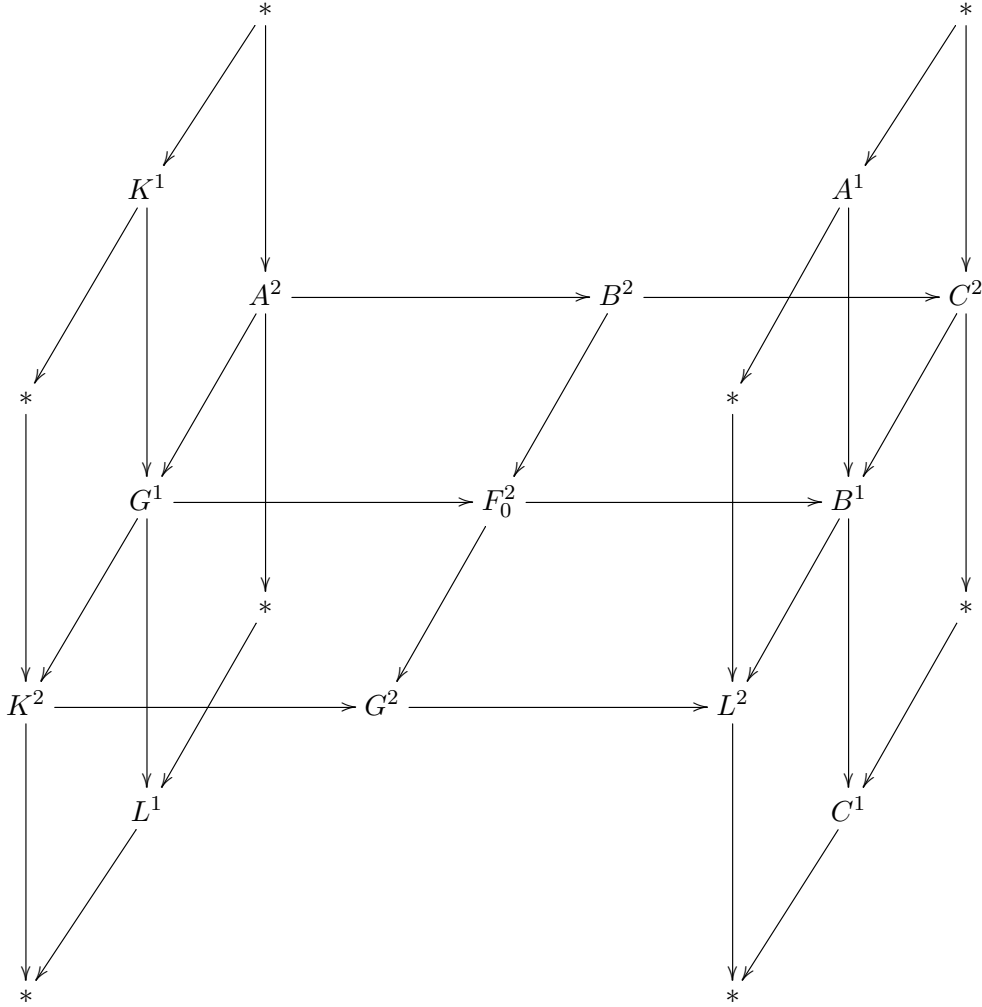
\xymatrix{ 
&  & \ast\ar[ddd]\ar[ddl] &  &  & & & & \ast\ar[ddd]\ar[ddl]\\
& & & & & & & & \\
 & K^1\ar[ddd]\ar[ddl] &  &  & & & & A^1\ar[ddd]\ar[ddl]&\\
&  & A^2\ar[ddd]\ar[ddl]\ar[rrr] &  &  &B^2\ar[rrr]\ar[ddl] & & & C^2\ar[ddd]\ar[ddl]\\
\ast\ar[ddd]& & & & & & \ast\ar[ddd]& & \\
 & G^1\ar[ddd]\ar[ddl]\ar[rrr] &  &  &F_0^2\ar[rrr]\ar[ddl] & & & B^1\ar[ddd]\ar[ddl]&\\
 & & \ast\ar[ddl]& & & & & & \ast\ar[ddl]\\
 K^2\ar[ddd]\ar[rrr] &  &  &G^2\ar[rrr] & & & L^2\ar[ddd]& &\\
& L^1\ar[ddl] &  &  & & & & C^1\ar[ddl]&\\
& & & & & & & & \\
\ast& & & & & & \ast& & \\
}
$$

For $n>2$ the elements of $\cC^n$ will appear as vertices of the $n+1$-dimensional hypercube in a diagram constructed in a recursive manner as above. 

Note that modulo $\Coh_\delta$ all components of $\cC^n$ vanish or are $(d,\delta+1)$-semistable. Note also that for any component $E$ of $\cC^n$, we have $0\le \cycle_d(E)\le\cycle_d(F_0)$. 
In fact the sum of these cycles over all components of $\cC^n$ equals $\cycle_d(F_0)$. Thus there exists some threshold $n_0\in \N$, such that the set of $d$-cycles of components of $\cC^n$ is constant for $n\ge n_0$. 
For $n>n_0$ the decomposition into building blocks from $\cC^n$ of  components $E$ of $\cC^{n_0}$ shows that $B^n$ cuts off sub-objects $E'$ in such components $E$ and these sub-objects are to be used in the reconstruction of $B^n$ itself; in fact they will be the quotients of a suitable filtration of $B^n$. If $E$ is $d$-dimensional then such a sub-object either vanishes or is pure $d$-dimensional since $B^n$ is pure $d$-dimensional itself. In this second case $\cycle_d(E')=\cycle_d(E)$ by our assumption on   $\cC^{n_0}$, and in particular $E/E'\in\Coh_\delta(X)$ and 
$\deg_\delta(E/E')=\deg_\delta(\cycle_\delta(E/E'))\ge 0$. If $E$ is not $d$-dimensional then $E$ has at most dimension $\delta$ and we have   $\deg_\delta(E/E')=\deg_\delta(\cycle_\delta(E/E')\ge 0$ in this case too. Consider now a subsequence $(B^{n_k})_{k>1}$ of $(B^n)_{n\ge n_0}$ such that all its terms cut off non-zero sub-objects on the same sub-collection of $d$-dimensional components of $\cC^{n_0}$. It follows that the sequence $(\cycle_d(B^{n_k}))_{k>1}$ is constant. On the other hand using the above notations and taking sums over all components $E$ of $\cC^{n_0}$ we find
$\deg_\delta(B^{n_k})=\sum_E \deg_\delta(E')=\sum_E\deg_\delta(E)-\sum_E\deg_\delta(E/E')=\deg_\delta(F_0)-\sum_E\deg_\delta(\cycle_\delta(E/E'))=\deg_\delta(F_0)-\deg_\delta(\sum_E\cycle_\delta(E/E'))$. It follows that the degrees in dimension $\delta$ of the cycles $\sum_E\cycle_\delta(E/E')$ are bounded and thus by our assumption on the degree functions they may attain only a finite number of values when $k$ varies. This implies that the sequence $(\mu_{d,\delta}(B^{n_k}))_k$ is eventually stationary,  hence also $(\mu_{d,\delta}(B^{n}))_n$ is eventually stationary.

We continue now the proof of the Claim following again \cite{HL}. By the above we may assume that the sequence $(\mu_{d,\delta}(B^{n}))_n$ is even constant. Then the inequalities \eqref{ineq:2} show that $C^{n+1}\neq0$ for all $n$ and using \eqref{ineq:1} we further find  
$\mu_{d,\delta}(B^{n+1})= \mu_{d,\delta}(C^{n+1})= \mu_{d,\delta}(B^{n})$. Thus either $A^{n+1}$ vanishes or it is $(d,\delta)$-semistable with $\mu_{d,\delta}(A^{n+1})= \mu_{d,\delta}(B^{n+1})= \mu_{d,\delta}(C^{n+1})$. In the latter case  we would have $\mu_{d,\delta}(B^{n+1})=\mu_{d,\delta}(A^{n+1})\le \mu_{d,\delta}(G^{n})<\mu_{d,\delta}(B^{n})$ which is impossible. Hence and from diagram \eqref{diagr:3X3} we get $A^{n+1}=0$, $B^{n+1}\cong C^{n+1}\subset 
B^{n}$, $G^n\cong K^{n+1}\subset G^{n+1}$ for all $n\in\N$.
Moreover since the sequence $(\cycle_d(B^n))_n$ is eventually stationary we may as well suppose that it is constant. It follows that $L^{n+1}\in\Coh_{\delta-1}$. 
In particular the ascending sequence of pure $d$-dimensional sheaves $G^n$ is constant modulo $\Coh_{d-2}(X)$, thus their reflexive hulls $(G^n)^{DD}$ are all the same and in particular the ascending sequence $(G^n)_n$ of  subsheaves of    $(G^0)^{DD}$ is eventually stationary. We assume again for simplicity that this sequence too is constant. So the central vertical and horizontal exact sequences of diagram \eqref{diagr:3X3} are split. 
We will write from now on  $G:=G^n$, $B:=B^n$ and $Q^n:=F/F^n$. Via the splittings $F_0^n\cong B\oplus G$  the morphisms $F_0^{n+1}\to F_0^n$ from diagram \eqref{diagr:snake} become compositions $B\oplus G\to B\to B\oplus G$ of the natural projections and injections. Hence the cokernel of the composition $F_0^n\to F^{n-1}_0\to...\to F_0^0$ is isomorphic to $G$. All in all we obtain that $Q_0^n\cong G$. From the diagram
$$
\xymatrix{  
0 \ar[r] & F^{n+1} \ar[r]\ar[d] & F\ar[r]\ar[d] & Q^{n+1}\ar[r] \ar[d] & 0 \\
0 \ar[r]& F^{n} \ar[r] & F\ar[r] & Q^{n}\ar[r]  & 0 
}
$$
we see that $\Ker(Q^{n+1}\to Q^n)\cong\Coker(F^{n+1}\to F^n)\cong G$. We will next show by induction on $n$ that $Q^n$ is flat over $\cO_\C/\gm^n$. 
The assertion is clear for $n=1$ so we assume it true for $n$ and start proving it for $n+1$. For this we tensorize the exact sequence $$0\to G\to Q^{n+1}\to Q^n\to 0$$ by
$$0\to \gm/\gm^{n+1}\to\cO_\D/\gm^{n+1}\to \cO_\D/\gm\to0$$ over $\cO_\D$ to get
$$
\begin{gathered}
\xymatrix{  
 & 0 \ar[d] & 0 \ar[d] & 0 \ar[d]& \\ 
0 \ar[r]& G\otimes\gm/\gm^{n+1} \ar[r]\ar[d] & G\otimes\cO_\D/\gm^{n+1}  \ar[d]\ar[r] & G\otimes\cO_\D/\gm\ar[r] \ar[d] & 0 \\
& Q^{n+1}\otimes\gm/\gm^{n+1}  \ar[r]\ar[d] & Q^{n+1}\otimes\cO_\D/\gm^{n+1}  \ar[d]\ar[r] & Q^{n+1}\otimes\cO_\D/\gm\ar[r] \ar[d] & 0 \\
0\ar[r] &  Q^n\otimes\gm/\gm^{n+1}  \ar[r]\ar[d] & Q^n \otimes\cO_\D/\gm^{n+1} \ar[d]\ar[r] & Q^n\otimes\cO_\D/\gm\ar[r] \ar[d] & 0 \\
 & 0  & 0  & 0 & 
}
\end{gathered}
$$
where the exactness of the first two vertical sequences follows from the induction hypothesis and from the fact that $\gm/\gm^{n+1} $ is a $\cO_\C/\gm^n$-module. Thus the morphism 
$Q^{n+1}\otimes\gm/\gm^{n+1}  \to Q^{n+1}\otimes\cO_\D/\gm^{n+1}$ is injective  and the local flatness criterion \cite[Lemma 2.1.3]{HL} shows that $Q^{n+1}$ is a flat $\cO_\D/\gm^{n+1}$-module. This gives a projective system of maps over $\D$ from the spaces $\Spec(\cO_\D/\gm^{n})$ to the relative Douady space $D_{F/X\times\D/\D}$ over $\D$ of quotients of $F$ which may be seen as a formal section of the natural projection of germs $(D_{F/X\times\D/\D}, F\twoheadrightarrow Q^1)\to (\D,0)$. Such a section admits an analytic approximation to order one by Artin approximation \cite[Theorem 1.4.ii]{Artin68} showing that the projection $(D_{F/X\times\D/\D}, F\twoheadrightarrow Q^1)\to (\D,0)$ is surjective. 
This contradicts the semistability of the fibres of $F$ over $\D^*$.
\end{proof}

\begin{prop}\label{prop:dimension one} 
Let $X$ be a compact complex manifold and
let $F$ be a $\D$-flat family of $d$-dimensional coherent sheaves on $X$ whose fibres over $\D\setminus\{0\}$ are pure of dimension $d$. Then there exists a coherent subsheaf $F'\subset F$ coinciding with $F$ over $\D\setminus\{0\}$ and such that its fibre $F'_0$ over $0$ is also pure.
\end{prop}
\begin{proof}
We follow the strategy of proof of Theorem \ref{thm:dimension one} but we take this time $B_n:=T_{d-1}(F_0^n)$, the maximal subsheaf of $F_0^n$ of dimension at most $d-1$, cf. \cite[Definition 1.1.4]{HL}. Then $G^n:=F_0^n/B^n$ is pure and diagram \eqref{diagr:3X3} takes the form 
\begin{equation}
\begin{gathered}
\xymatrix{ 
 &  & 0 \ar[d] & 0 \ar[d]& \\ 
&  & B^{n+1} \ar[d]\ar[r]^\cong & B^{n+1} \ar[d] &  \\
0 \ar[r]& G^{n} \ar[r]\ar[d]^\cong & F^{n+1}_0 \ar[d]\ar[r] & B^{n}\ar[r] \ar[d] & 0 \\
0 \ar[r] & G^{n} \ar[r] & G^{n+1} \ar[d]\ar[r] & L^{n+1}\ar[r] \ar[d] & 0 \\
 &   & 0  & 0 & 
}
\end{gathered}
\end{equation}

We have  inequalities for associated $(d-1)$-cycles:
$$0\le\cycle_{d-1}(B^{n+1})\le\cycle_{d-1}(B^{n}),$$
hence the sequence $(\cycle_{d-1}(B^{n}))_n$ must be eventually stationary and we may assume that $\dim(L^n)\le d-2$ for all $n$. We immediately get then that the ascending sequence of subsheaves of $(G_0)^{DD}$ is eventually stationary and as before we assume that this sequence is constant and write 
$G:=G^n$, $B:=B^n$. The rest of the proof follows ad litteram the  proof of Theorem \ref{thm:dimension one} but for its last sentence where we get a contradiction to purity instead of semistability. 
\end{proof}

\section{Families of arbitrary dimension}\label{section:arbitrary-dimension}


We now turn our attention to the case of higher dimensional parameter  spaces.  As in the previous section we give a separate purity statement. This is the content of Proposition \ref{prop:dimension>1}. 
For the main result of the section a stronger assumption on the degree functions will be made which will guarantee the existence of a relative maximal destabilizing subsheaf and in particular that semistability is a Zariski open property in flat families of coherent sheaves, see Section \ref{section:degrees}. For simplicity we will only consider smooth
parameter spaces. This is no significant restriction of generality since the property stated is valid after a birational base change.

\begin{thm}\label{thm:dimension>1} 
Let  $X$ be a compact complex manifold endowed with a strong $(d,d')$-degree system induced by a system of strictly positive $\partial\bar\partial$-closed differential forms, where $0\leq d'\leq d$. Let further $F$ be a  family over $S$ of $d$-dimensional sheaves on $X$, where $S$ is a connected smooth parameter space.  Suppose that general fibres of $F$ are  $(d,d')$-semistable and let $Z\subset S$ be the union of the non-flatness locus of $F$ with the closed analytic subset of $S$ parametrizing non-$(d,d')$-semistable sheaves. Let further $K\subset S$ be a compact subset. Then there exist a proper modification $S'\to S$ and a coherent sheaf $F'$ on $X\times S'$ such that $F'$ is flat over $S'$, all fibres of $F'$ over $K\times_SS'$ are $(d,d')$-semistable and $F'$ coincides with $F_{S'}$ over $(S\setminus Z)\times_SS'$.
\end{thm} 
\begin{proof}
It is clear that we only need to deal with the finitely many irreducible components of $Z$ which meet the compact set $K$. 
In the sequel we will assume for simplicity of notation that all irreducible components of $Z$ meet $K$. 
The idea of the proof is on one hand  to try to reduce the dimension of the bad set $Z$ for a suitable subsheaf of $F$ constructed by a similar procedure to that which is used in the proof of Theorem \ref{thm:dimension one}; on the other hand it will be convenient to work in the case when $Z$ is a simple normal crossings divisor and $F$ is flat over $S$. We may reduce ourselves to this situation by repeatedly blowing up $S$ at smooth centers by Hironaka's flattening theorem \cite{Hir75}.  
Since under this requirement the dimension of $Z$ is maximal, we introduce a "badness index" $b$ in order to control the induction process in the following way: For any irreducible component $Z_i$ of $Z$ we say that a proper holomorphic map $\pi_i:Z_i\to B_i$ is {\em good (for $F$)} if the restrictions of $F$ to the fibres of $\pi_i$ are {\em isotrivial}, i.e. the fibres of these restricted families are two by two isomorphic. 
Then we set $b_i$ to be lowest possible dimension of a base $B_i$ of such a good map, and $b$ to be the maximum among the $b_i$.  If $Z$ is empty we  assign to $F$ a negative badness index, $b=-\infty$ say.
The strategy will be the following: we start with the flat family $F$ whose non-semistable locus is a divisor with simple normal crossings $Z=\sum_iZ_i$ and with badness index $b$.
From this data we produce a subfamily $F'\subset F$, with $F'_{S\setminus Z}= F_{S\setminus Z}$ and with strictly smaller bad locus. 
If this bad locus is empty, then $F'$ is the family we were looking for and we stop.
If not,  after a suitable proper modification $S'\to S$ the pull-back of this non-semistable locus becomes a divisor with simple normal crossings $D'$ on $S'$ with  badness index $b'$ for the family $F'_{S'}/\Tors_{S'}(F'_{S'})$ and such that $b'<b$. We work now with the family $F'_{S'}/\Tors_{S'}(F'_{S'})$ which may be supposed in addition to be flat over $S'$, by Hironaka's flattening theorem \cite{Hir75} again. It is clear that this process eventually stops.

The proof of the existence of the desired subfamily $F'\subset F$ on $S$ will follow the same path as in the one-dimensional case by descending induction on $\delta$. The corresponding claim will be 
\begin{claim}\label{Claim>1}
Let $X$, $F$ be as in the theorem's statement and moreover such that $F$ is flat over $S$ and $Z$ is a divisor with simple normal crossings. Suppose that   $d\ge\delta \ge d'$.  the general fibres of $F_Z$ are  $(d,\delta+1)$-semistable for some $\delta$ with  $d\ge\delta \ge d'$. Then 
there exists  a subsheaf $F'\subset F$  with $F'_{S\setminus Z}=F_{S\setminus Z}$ with $(d,\delta)$-semistable fibres over general points of $Z$ and with non-flatness locus which is  nowhere dense in $Z$. 
\end{claim}
We take first a relative maximal (semi)destabilizing subsheaf $B^0$ of $F_Z$ and  put $G^0:=F_Z/B^0$ and $F^1:=\Ker(F\to G^0)$. Note that $B^0_s$ is the maximal $(d,\delta)$-(semi)destabilizing subsheaf only for general points $s\in Z$. Flatness of $B^0$ and $G^0$ likewise only holds generically over $Z$. But we can work at such general points and see that the whole proof of Claim \ref{Claim} goes through in this new  relative setting. (We will consider of course relative support cycles for the components of the collections $\cC^n$ this time. We will also use the corresponding relative statements at each moment of the proof, such as \cite[Theorem 1.3]{Artin68} for instance.) In this way we obtain a subsheaf $F'\subset F$  with $F'_{S\setminus Z}=F_{S\setminus Z}$ and with $(d,\delta)$-semistable fibres over general points of $Z$. 
At each step of the proof the non-flatness locus of the sheaves $F^n$ is nowhere dense in $Z$. Indeed, assuming this to be true for $F^n$, we check it for $F^{n+1}$ 
by using an analogue of diagram \eqref{diagr:snake} around a point of $Z$ where both $F^n$ and $G^n$ are flat: 
\begin{equation}
\label{diagr:snake>1}
\begin{gathered}
\xymatrix{ 
 & 0 \ar[d] & 0 \ar[d] & Tor_1^{\cO_{S}}(G^n,\cO_{Z}) \ar[d]& \\ 
& F^{n+1}(-Z) \ar[r]\ar[d] & F^{n+1} \ar[d]\ar[r] & F^{n+1}_{Z}\ar[r] \ar[d] & 0 \\
0 \ar[r]&  F^n_{Z}(-Z) \ar[r]\ar[d] & F^n_{Z} \ar[d]\ar[r] & F^n_{Z}\ar[r] \ar[d] & 0 \\
Tor_1^{\cO_{Z}}(G^n,\cO_{Z}) \ar[r]^{\cong} &  G^n(-Z) \ar[r]^0\ar[d] & G^n \ar[d]\ar[r]^{\cong} & G^n\ar[r] \ar[d] & 0 \\
 & 0  & 0  & 0 & 
}
\end{gathered}
\end{equation}
from which we immediately obtain exact sequences
$$ 0\to G^n(-Z)\to F^{n+1}_{Z}\to F^n_{Z}\to G^n\to0 $$
and
$$
0 \to F^{n+1}(-Z) \to F^{n+1}\to F^{n+1}_{Z}\to 0.$$
From the first sequence we infer that $F^{n+1}_{Z}$ is flat over $Z$ around the chosen point and from the second combined with \cite[Corollaire 5.1.4]{BaSt} that $F^{n+1}$ is flat over $S$ around the chosen point again.  This proves Claim \ref{Claim>1}. 

If the bad locus $Z'$ of $F'$ on $S$ is not empty, we perform  a proper modification $S'\to S$ on $S$ so that $F'':=F'_{S'}/\Tors_{S'}(F'_{S'})$ is flat over $S'$ with divisorial bad locus $Z''$. Let $b$, $b'$ and $b''$ be the badness indices of $F$, $F'$ and $F''$, respectively. 
Let $Z_i\subset Z$ be an irreducible component of $Z$ and  $\pi_i:Z_i\to B_i$ be a good map for $F$. 
By construction of $F'$ it follows that the restrictions of the family $F'$ to  the fibres of $Z'_i\to B_i$ are isotrivial, hence the intersection  $Z_i\cap Z'$ fibres over a proper Zariski subset $B'_i$ of $B_i$. This shows that $b'<b$. When passing to the pair $(S',F'')$ it is clear that good maps for $F''$ are obtained from good maps for $F'$ by composing with $S'\to S$. In particular one has $b''\le b'<b$.  
 We have thus completed the induction argument and the proof of the theorem.    
\end{proof}

As in the case of one dimensional bases the previous arguments adapt to yield the following

\begin{prop}\label{prop:dimension>1} 
Let $X$ be a compact complex manifold 
and let $F$ be a  family over $S$ of $d$-dimensional sheaves on $X$, where $S$ is a connected smooth parameter space.  Suppose that general fibres of $F$ are  pure and let $Z\subset S$ be the union of the non-flatness locus of $F$ with the closed analytic subset of $S$ parametrizing non-pure sheaves. Let further $K\subset S$ be a compact subset. Then there exist a proper modification $S'\to S$ and a coherent sheaf $F'$ on $X\times S'$ such that $F'$ is flat over $S'$, all fibres of $F'$ over $K\times_SS'$ are pure and $F'$ coincides with $F_{S'}$ over $(S\setminus Z)\times_SS'$.
\end{prop} 

\section{Application to  moduli spaces of semistable sheaves} \label{section:application}

In this section we give an application of Theorem \ref{thm:dimension>1} to compactness of moduli spaces of semistable sheaves. For such a result to hold one needs boundedness for the class of sheaves one is considering. This is typically attained  by fixing the topological type or the Hilbert polynomial of these sheaves. This is the way to think of property $\cP$ in the statement below. The moduli spaces we will consider will be defined according to \cite[Definition 4.1.1]{HL} adapted to the complex analytic set-up, i.e. they will corepresent the functor associating to an analytic space $S$ the set of isomorphism classes of $S$-flat families of semistable sheaves with the property $\cP$  on $X$.   

\begin{thm}\label{thm:application} 
Let $X$ be a compact complex manifold endowed with a weak $(d,d')$-degree system and let $\cP$ be an open and closed property on flat families of $d$-dimensional coherent sheaves on  $X$. Suppose that the corresponding semistability property satisfies the following properties:
\begin{enumerate}
\item Zariski openness in flat families of coherent sheaves,
\item boundedness when restricted to the class of sheaves with the property $\cP$,
\item the conclusion of Theorem \ref{thm:dimension>1}  when restricted to the class of sheaves with the property $\cP$,
\item the existence of a coarse moduli space $M^{ss}_\cP$ for semistable sheaves with the property $\cP$.  
\end{enumerate}
Then
$M^{ss}_\cP$ is quasi-compact.
\end{thm} 
\begin{proof}
Let $K\subset S$, $F$ be a compact subset of a smooth complex space and a family of coherent sheaves on $X$ over $S$ giving the boundedness of the set of isomorphism classes of semistable sheaves having property $\cP$ on $X$, cf. \cite{TomaLimitareaI}, \cite{TomaLimitareaII}. 
By restricting $S$ to a finite number of its connected components, we may suppose that all the sheaves in the corresponding family over $S$ have the property $\cP$. Let $S^{ss}\subset S$ be the open subset which parametrizes semistable sheaves, let $D:=S\setminus S^{ss}$ and let $S'\to S$ be the proper modification given by Theorem \ref{thm:dimension>1}. Then the family $F'$ given by the conclusion of Theorem \ref{thm:dimension>1}, the universal property of $M^{ss}_\cP$ and the choice of $K$ and $F$ show the existence of a surjective morphism $K\times_SS'\to M^{ss}_\cP$. This proves our statement. 
\end{proof}

In order to get compactness (or properness) for such a moduli space, one needs to show that $M^{ss}_\cP$ is separated. This can be checked by different methods in the following two examples. For the sake of completeness a separation result will be proved for $(n,0)$-semistable sheaves in Section \ref{section:separatedness}.  

We end this Section with two examples where our methods apply. 
\begin{Ex} Let $X$ be an $n$-dimensional complex projective manifold, $\omega$ a K\"ahler class on $X$, $c_j\in H^{2j}(X,\Z)$, $1\le j\le n$, and $\cP$ the property for a coherent sheaf on $X$ of having rank two and Chern classes equal to $c_j$, $1\le j\le n$. Consider the semistability condition given by the $(n,0)$-degree system induced by $\omega$ as described in Section \ref{section:degrees}. Then $M^{ss}_\cP$ is compact by Theorem \ref{thm:application}. Indeed,
condition (1) was proved in \cite{TomaLimitareaII}, condition (2) in \cite{GrebToma}, condition (3) in this paper and condition (4) as well as separatedness in \cite{GTgood}. In fact, by \cite{GTgood} the moduli space $M^{ss}_\cP$ is an algebraic space over $\C$ so condition (3) of Theorem \ref{thm:application} may be replaced by a requirement of properness over $1$-dimensional bases as provided by Theorem \ref{thm:dimension one}.
\end{Ex}
\begin{Ex}
Let $(X,\omega)$ be a compact K\"ahler surface, $r\in\Z_{>0}$, $c_1\in \NS(X)$, $c_2\in H^{4}(X,\Z)$, $1\le j\le n$, and $\cP$ the property for a coherent sheaf on $X$ to be of rank $r$ and to have Chern classes equal to $c_1$, $c_2$, respectively. Consider the semistability condition given by the $(2,1)$-degree system induced by $\omega$, i.e. the usual slope-semistability with respect to $\omega$. Suppose moreover that there exist no properly semistable sheaves with property $\cP$, i.e. all semistable sheaves with $\cP$ are stable. This will be the case if $c_1$ and $r$ are relatively prime and if $\omega$ is a generic class in the K\"ahler cone of $X$,  \cite[Section 2.2]{PeregoToma}. Then Theorem \ref{thm:application} may be applied also to this situation. Conditions (1) and (3) are checked as in the previous example. Condition (2) makes the object of \cite{BuTeTo2}. Finally, the moduli of stable sheaves exists as a Hausdorff open subspace of the moduli space of simple sheaves, whose existence was proved in \cite{KosarewOkonek}. For K\"ahler surfaces this example generalizes the compactness result  \cite[Theorem 4.3 (c)]{TomaDocumenta} by different methods. 
\end{Ex}

\section{A criterion for separatedness}\label{section:separatedness}

Here we prove a separation result in the complex analytic set-up, which for the sake of simplicity will only be formulated for $(n,0)$-semistable sheaves. In this special case any semistable sheaf admits a Jordan-H\"older filtration and a unique Jordan-H\"older graduation up to isomorphism, cf. \cite[Proposition 2.8]{GTgood}. This property is not in general true in the case of $(n,n-1)$-semistability which is the situation considered by Langton. This is the reason why the conclusion of his Theorem (1) \cite[p. 99]{Langton} cannot be as sharp as ours.  
Recall that if a moduli space $M^{ss}_\cP$ exists, semistable sheaves with  isomorphic Jordan-H\"older graduations will give the  same point in $M^{ss}_\cP$, \cite[Lemma 4.1.2]{HL}. It is therefore natural to formulate a separation criterion in terms of Jordan-H\"older graduations.

Let $X$ be an $n$-dimensional compact analytic space endowed with a weak $(n,0)$-degree system. For the induced semistability notion the following three results are proved as in the case of Gieseker-Maruyama semistability, cf.~\cite[Prop.~1.2.7]{HL}, \cite[Prop.~3.1]{Seshadri} and \cite[Section 9.3]{LePotier_Lectures}, and finally \cite[Prop.~1.5.2]{HL}, respectively. 

\begin{lemma}\label{lemma:morphisms}
Let $E$ and $E'$ be  semistable sheaves on $X$  and let $\phi :E\to E' $ be a non-zero morphism of $\cO_X$-modules. Then $\mu(E)\le \mu({E'})$. If equality holds, then $\im(\phi)$ is semistable and $\mu({\im(\phi)})=\mu(E)=  \mu({E'})$. If moreover the $n$-degree of $\im(\phi)$ coincides with the $n$-degree of $E$ or with the $n$-degree of $E'$ then $\im(\phi)$ is isomorphic to $E$ or to $E'$ respectively.
\end{lemma}

\begin{lemma}
\label{lemma:abelian_cat}
The full subcategory $\Coh^{ss}{(X;\deg_n,...,\deg_0;\mu)}$ of the category of coherent sheaves on $X$, whose objects are the semistable sheaves with fixed slope vector $\mu$ and the zero-sheaf, is abelian, noetherian and artinian. 
\end{lemma}

\begin{lemma}[Jordan-H\"older filtrations]\label{lemma:JH}
 Any semistable sheaf on $X$ has a Jordan-H\"older filtration  in the sense of \cite[Def.~1.5.1]{HL}. The associated graded sheaf is unique up to isomorphism.
\end{lemma}

We will denote by $\gr(E)$
the Jordan-H\"older graduation of a semistable sheaf $E$. 

We can now state the main result of this section.

\begin{thm}\label{thm:separation} 
Let  $X$ be an $n$-dimensional reduced compact complex space endowed with an $(n,0)$-degree system and let $E$ and $F$ be two $\D$-flat families of semistable $n$-dimensional sheaves on $X$ which are fibrewise isomorphic over $\D^*$. Then $$\gr(E_0)\cong\gr(F_0).$$
\end{thm}
\begin{proof}
We start by showing that after possibly shrinking $\D$ around $0$ a morphism $\phi:E\to F$ exists inducing isomorphisms in all fibres over $\D^*$. For this, we denote by $p:X\times\D\to\D$ the second projection, and by 
$\cH:=\cH om_{\cO_{X\times\D}}(p;E,F)=p_*\cH om{\cO_{X\times\D}}(E,F)$ the relative $\Hom$-sheaf with respect to $p$. 
This is a coherent sheaf on $\D$ of rank at least $1$. Let $s_{1,0},..., s_{m,0}$ be generators of the stalk $\cH_0$ as $\cO_{\D,0}$-module. These elements have representatives $s_1,..., s_m\in \cH(U)$ over a small neighbourhood $U$ of $0$ in $\D$ which will also generate any other stalk $\cH_t$, for $t\in U$. For some $t\in U\setminus\{0\}$ we may consider thus a combination $\phi$ of the sections $s_1,..., s_m$ over $\cO_\D(U)$ which restricts to an isomorphism at $t$. Then $\phi$ will restrict to an isomorphism at any point in a small neighbourhood of $0$ except possibly at $0$. We shrink $\D$ accordingly.

We denote by $\phi_0:E_0\to F_0$ the restriction of $\phi$ to the central fibre and by $E'$ the kernel of the composition 
$$E\to F\to F_0.$$  
Note that by continuity of the degree functions we have 
$\deg_p(E_0)=\deg_p(F_0)$ for all $p$.
If $\phi_0$ is an isomorphism the statement of the Theorem is clear. This will happen if and only if $\Coker (\phi_0)=0$ if and only if $\Coker (\phi)=0$ as follows from the following commutative diagram with exact rows and columns

$$
\begin{gathered}
\xymatrix{  
 & 0 \ar[d] & 0 \ar[d] & 0 \ar[d]& \\ 
0 \ar[r]& E' \ar[r]\ar[d] & F\otimes_\D\gm  \ar[d]\ar[r] & 
C'\ar[r] \ar[d] & 0 \\
0 \ar[r]& E\ar[r]\ar[d] & F \ar[d]\ar[r] & \Coker (\phi) \ar[d]\ar[r] & 0 \\
0\ar[r] & \im(\phi_0)  \ar[r]\ar[d] & F_0 \ar[d]\ar[r] & \Coker (\phi)\otimes_\D\cO_\D/\gm\ar[r] \ar[d] & 0 \\
 & 0  & 0  & 0 & 
}
\end{gathered}
$$
where $C':=\Ker(\Coker (\phi)\to\Coker (\phi)\otimes_\D\cO_\D/\gm$ and, as before, $\gm$ is the ideal sheaf of $\{0\}$ in $\D$. 
Note that the $\cO_X$-module $\Coker (\phi_0)\cong \Coker (\phi)\otimes_\D\cO_\D/\gm$ is semistable hence pure. It thus vanishes if and only if its $n$-degree vanishes. 
Supposing that this is not the case, we get
\begin{equation}
\label{inequality}
\deg_n(C')<\deg_n( \Coker (\phi)),
\end{equation}
where these degrees are computed for the associated $\cO_X$-modules on the regular part of $X$.
Note that $C:=Coker(\phi)$ and $C'$ are supported on $X\times0:=X\times\{0\}$ and hence are $\cO_{X\times \D}/p^*(\gm)^m$-modules for $m$ high enough.

We replace our families $E$, $F$ by $E'$ and $F':=F\otimes_\D\gm$ respectively and the exact sequence
$$0\to E\to F\to C\to 0$$
by
$$0\to E'\to F'\to C'\to 0.$$ In view of the inequality \eqref{inequality}, iteration of this procedure must end after finitely many steps with an isomorphism $E^{(l)}\to F^{(l)}$ and the theorem will follow if we can check that
$\gr(E_0)\cong\gr(E'_0)$ and $\gr(F_0)\cong\gr(F'_0)$. The second equality is clear since $F'_0\cong F_0$. It remains to establish the first one. 

Tensoring the sequence 
$$0\to E'\to E\to  \im(\phi_0)\to 0$$
by $\cO_\D/\gm$ over $\cO_\D$, we get
$$0\to\cT or_1^{\D}(\im(\phi_0), \cO_\D/\gm)\to E'_0\to E_0\to \im(\phi_0)\to0.$$
On the other hand we get $\cT or_1^{\D}(\im(\phi_0), \cO_\D/\gm)\cong \im(\phi_0)\otimes_\D\cO_\D/\gm$ by tensoring the sequence $0\to \gm\to \cO_\D\to\cO_\D/\gm\to0$ by $\im(\phi_0)$. Thus $\Ker(E'_0\to E_0)\cong \im(\phi_0)\otimes_\D\cO_\D/\gm\cong \im(\phi_0)$ and we get  devissages  by semistable sheaves 
$$0\to \Ker(\phi_0)\to E_0\to  \im(\phi_0)\to 0,$$
$$0\to \im(\phi_0)
\to E'_0\to  \Ker(\phi_0)\to 0$$
for $E_0$ and for $E'_0$ on $X$, hence
 $$\gr(E_0)\cong\gr(E'_0)$$
 and the Theorem is proven.
\end{proof}


\def\cprime{$'$} \def\polhk#1{\setbox0=\hbox{#1}{\ooalign{\hidewidth
  \lower1.5ex\hbox{`}\hidewidth\crcr\unhbox0}}}
  \def\polhk#1{\setbox0=\hbox{#1}{\ooalign{\hidewidth
  \lower1.5ex\hbox{`}\hidewidth\crcr\unhbox0}}}
\providecommand{\bysame}{\leavevmode\hbox to3em{\hrulefill}\thinspace}
\providecommand{\MR}{\relax\ifhmode\unskip\space\fi MR }
\providecommand{\MRhref}[2]{%
  \href{http://www.ams.org/mathscinet-getitem?mr=#1}{#2}
}
\providecommand{\href}[2]{#2}

\medskip
\medskip
\begin{center}
\rule{0.4\textwidth}{0.4pt}
\end{center}
\medskip
\medskip
\end{document}